\documentclass[11pt]{amsart}

\usepackage{latexsym}
\usepackage{amsmath}
\usepackage{amsfonts}
\usepackage{amssymb}
\usepackage{enumitem}
\usepackage{blkarray}
\usepackage {tikz}
\usepackage{color}
\usepackage{mathtools}
\usetikzlibrary {positioning}
\usepackage{blkarray, bigstrut}
\usetikzlibrary{arrows,shapes}
\usepackage{color}

\usepackage[noend]{algorithmic} 
\usepackage{algorithm,caption}
\algsetup{indent=2em}

\newtheorem{theorem}{Theorem}[section]
\newtheorem*{theorem*}{Theorem}
\newtheorem*{corollary*}{Corollary}

\newtheorem{corollary}[theorem]{Corollary}
\newtheorem{proposition}[theorem]{Proposition}

\newtheorem{sublemma}{}[theorem]

\theoremstyle{definition}

\theoremstyle{remark}

\numberwithin{equation}{section}
\usepackage{stackengine}

\stackMath

\begin{document}

\title[$k$-loose elements]{$k$-loose elements and $k$-paving matroids}

\author{Jagdeep Singh}
\address{Department of Mathematics and Statistics\\
Mississippi State University\\
Starkville, Mississippi 39762}
\email{singhjagdeep070@gmail.com}

\date{\today}

\begin{abstract}
For a matroid of rank $r$ and a non-negative integer $k$, an element is called $k$-loose if every circuit containing it has size greater than $r-k$. Zaslavsky and the author characterized all binary matroids with a $1$-loose element. In this paper, we establish a sharp linear bound on the size of a binary matroid, in terms of its rank, that contains a $k$-loose element. A matroid is called $k$-paving if all its elements are $k$-loose. Rajpal showed that for a prime power $q$, the rank of a $GF(q)$-matroid that is $k$-paving is bounded. We provide a bound on the rank of $GF(q)$-matroids that are cosimple and have two $k$-loose elements. Consequently, we strengthen the result of Rajpal by providing a bound on the rank of $GF(q)$-matroids that are $k$-paving. Additionally, we provide a bound on the size of binary matroids that are $k$-paving. 
\end{abstract}

\maketitle

\section{Introduction}
The notation and terminology follow \cite{text} unless stated otherwise. All graphs and matroids considered here are finite and simple. For a matroid of rank $r$ and a non-negative integer $k$, an element $t$ of $M$ is called \textbf{$k$-loose} if every circuit of $M$ that contains $t$ has size greater than $r-k$. 
A matroid $M$ is called \textbf{$k$-paving} if every element of $M$ is $k$-loose. Note that a matroid $M$ is paving if and only if $M$ is $1$-paving. Acketa \cite{acketa} provided a characterization of all $1$-paving matroids that are binary, while Oxley \cite{ternary_pav} determined all ternary $1$-paving matroids. Zaslavsky and the author characterized all binary matroids with a $1$-loose element \cite{sinzas}. 
In Section 2, we show that the size of a binary matroid that contains a $k$-loose element is linear in terms of its rank. In particular, we provide the following sharp bound.

\begin{theorem}
\label{size_bound_one_k_loose}
For $r \geq 5$ and $0 \leq k \leq \frac{r-2}{3}$, let $M$ be a simple binary matroid of rank $r$ with no coloops and a $k$-loose element. Then $|E(M)| \leq 2^k(r-k+1)$.
\end{theorem}

We say a matroid $M$ is a $GF(q)$-matroid if $M$ is representable over the field of $q$ elements. Section 3 considers $GF(q)$-matroids that have two $k$-loose elements. We show that the rank of such a matroid is bounded unless the two $k$-loose elements form a cocircuit of size two. The following is the precise statement.

\begin{theorem}
\label{rank_bound_two_loose}
Let $M$ be a simple $GF(q)$-matroid with no coloops that has two $k$-loose elements $e$ and $f$. Then $r(M) \leq (q+1)(k-1)+2q$, or $\{e,f\}$ is a cocircuit of $M$.
\end{theorem}

An immediate consequence of Theorem \ref{rank_bound_two_loose} is the following.

\begin{corollary}
\label{rank_bound_paving}
Let $M$ be a simple $GF(q)$-matroid with no coloops that is $k$-paving. Then $M$ is a circuit or $r(M) \leq (q+1)(k-1)+2q$. 
\end{corollary}

Rajpal \cite[Proposition 8]{rajpal2} demonstrated the existence of a rank bound for $k$-paving matroids representable over $GF(q)$.
Corollary \ref{rank_bound_paving} strengthens this result by giving an explicit rank bound for $GF(q)$-matroids that are $k$-paving. Observe that paving matroids are precisely $1$-paving matroids. Therefore by Corollary \ref{rank_bound_paving}, a paving matroid representable over $GF(q)$ has rank at most $2q$, unless it is a circuit.

A binary $k$-paving matroid is \textit{maximal} if it is not a proper minor of any other binary $k$-paving matroid. Rajpal \cite{rajpal2} established a connection between maximal binary $k$-paving matroids and Reed-Muller codes, proving the following.

\begin{proposition}
For $k \geq 1$, let $G$ be a generator matrix of the Reed-Muller code $R(1,k+2)$. Then the vector matroid of $G$ is a maximal binary $k$-paving matroid.
\end{proposition}

Additionally, Rajpal determined all binary $2$-paving matroids. In Section 4, we provide a bound on the size of binary $k$-paving matroids.

\begin{theorem}
\label{size_bound_paving}
Let $M$ be a simple binary $k$-paving matroid of rank $r$, where $M$ is not a circuit and $r \geq k+4$, and let $t = \lfloor \frac{3k+1-r}{2} \rfloor$. Then  $r(M) \leq 3k+1$, and $$|E(M)| \leq (r+1) + \sum_{i=0}^t \binom{k}{i}.$$
\end{theorem}

Note that every simple binary matroid with rank at most $k+2$ is $k$-paving, and every simple triangle-free binary matroid of rank $k+3$ is $k$-paving. This includes binary affine matroids of rank $k+3$. To demonstrate the assistance Theorem \ref{size_bound_paving} provides in finding all $k$-paving binary matroids, we apply it to $3$-paving binary matroids and obtain the following result.

\begin{corollary}
\label{3_paving}
Let $M$ be a simple binary $3$-paving matroid of rank at least seven. Then, $r(M) \leq 10$ and $|E(M)| \leq 13$. In particular, the size of $M$ is at most $12, 13, 11,$ and $12$ when the rank of $M$ is $7, 8, 9,$ and $10,$ respectively. 
\end{corollary}

\section{$k$-loose elements in binary matroids}

Theorem \ref{size_bound_one_k_loose} provides a sharp linear bound on the size of a binary matroid containing a $k$-loose element, in terms of its rank. Specifically, if a binary matroid $M$ has a $1$-loose element, then the size of $M$ is at most $2r$. If $M$ has a $2$-loose element, then the size of $M$ does not exceed $4r-4$.

\begin{proof}[Proof of Theorem \ref{size_bound_one_k_loose}]

Let $e$ be a $k$-loose element of $M$. Note that if $k = 0$, then by \cite[Lemma 3.1]{oxjag}, $M$ is a circuit, and the result follows. Therefore $k \geq 1$. Suppose that $e$ is not $(k-1)$-loose. It follows that there is a basis $B = \{e_1, \ldots, e_r\}$ of $M$ such that, in the standard binary representation $P = [I_r|Q]$ of $M$ with respect to $B$, the column $e^P$ labeled by $e$  has its first $k$ entries equal to zero, while all other entries are non-zero. We call the set of indices of the non-zero entries of a column $f^P$ in $Q$ the \textit{support} of $f^P$, denoted as $\text{supp}(f^P)$. The entries of $f^P$, excluding the first $k$ entries, are called the \textit{root entries} of $f^P$, and the set of indices of the non-zero root entries of $f^P$ is the \textit{root support} of $f^P$. Let $Q-e^P$ denote the columns in $Q$ excluding $e^P$.

\begin{sublemma}
\label{zero}
For a column $f^P$ in $Q-e^P$, the sum of the number of zeroes in the first $k$ entries of $f^P$ and the number of non-zero root entries of $f^P$ is at most $k+1$. 
\end{sublemma} 

Suppose otherwise. Then $e^P + f^P$ has at least $k+2$ zeroes, so at most $r-(k+2)$ columns from $I_r$ can be added to $e^P + f^P$ to obtain a column with all entries equal to zero. It follows that $e$ is in a circuit of size at most $r-k$, a contradiction. Therefore \ref{zero} holds. 

Let $F$ be a set of all columns in $Q-e^P$ such that, for any pair $f^P, g^P$ in $F$, the sum $f^P + g^P$ has its first $k$ entries equal to zero. It is clear that the first $k$ entries of all columns in $F$ are identical.

\begin{sublemma}
\label{one}
Let $f^P, g^P$ be columns in $F$. Then, the symmetric difference of their supports, $ supp(f^P) \bigtriangleup supp(g^P)$, contains at most two elements.  
\end{sublemma}

Suppose not. Then $f^P + g^P$ has its first $k$ entries equal to zero and contains at least three non-zero entries. It follows that $e^P + f^P + g^P$ has at least $k+3$ zeroes so it can be combined linearly over GF(2) with at most $r-(k+3)$ columns in $I_r$ to obtain the column that has all the entries zero. Therefore $e$ is in a circuit of size at most $r-k$, a contradiction. Thus \ref{one} holds.

\begin{sublemma}
\label{two}
$|F| \leq r-k+1$. Moreover, if the columns in $F$ have exactly one non-zero entry in the first $k$ entries, then $|F| \leq r-k$.
\end{sublemma}

Choose a column $f^P$ in $F$ such that the root support of $f^P$, say $S$, has the largest possible size. Suppose $|S| = s$. It follows by \ref{zero} that $s \leq k+1$. In view of \ref{one}, a column $g^P$ in $F-f^P$ can fall into one of the following three categories, namely, 

\begin{enumerate}[label=(\roman*)]

    \item $\text{supp}(g^P) \subseteq \text{supp}(f^P)$ with $|\text{supp}(f^P) - \text{supp}(g^P)| = 1$, 
    \item $\text{supp}(g^P) \subseteq \text{supp}(f^P)$ with $|\text{supp}(f^P) - \text{supp}(g^P)| = 2$, or 
    \item $\text{supp}(g^P) \not \subseteq \text{supp}(f^P)$. 
\end{enumerate}

We call these columns type A, type B, and type C, respectively. 
By \ref{one}, it follows that the number of type B columns is at most $ \text{max}\{3, s-1\}$. Note that if $s \leq 1$, there is no type B column. If $s=2$, there is at most one type B column. It follows that the number of type B columns does not exceed $s$.

First, suppose that $F$ contains a type C column. We say an entry of a column $g^P$ is in $S$ if the index of the entry \textit{is in $S$}. Similarly, an entry of $g^P$ is \textit{outside $S$} if its index is not in $S$. Since the size of the root support of any column in $F$ is at most $s$, it follows from \ref{one} that a type C column has precisely one non-zero root entry outside $S$ and exactly one zero entry in $S$.  Observe that if $F$ contains two type C columns such that their corresponding entries in $S$ are not identical, then by \ref{one}, all their corresponding entries outside $S$ must be identical. In this case, the number of type C columns is at most $s$. On the other hand, if all type C columns in $F$ have identical entries in $S$, then the number of type C columns is at most $r-(s+k)$. Therefore, the number of columns of type C in $F$ is at most $\text{max}\{s, r-(s+k)\}$. Given that $s \leq k+1$ and $k \leq \frac{r-2}{3}$, it follows that the number of type C columns in $F$ is at most $r-(s+k)$. Since $F$ contains a type C column, it is immediate from \ref{one} that there is at most one type A column in $F$. Furthermore, if a type A column exists, then we have at most $s-1$ type B columns. It follows that the sum of the number of type A columns and the number of type B columns does not exceed $s$. Therefore the size of $F$ is at most $r-(s+k) + s + 1 = r-k+1$.

Now, suppose that no type C column exists. By \ref{one}, it follows that if a type B column exists, then the number of type A columns is at most $2$ so $F$ has size at most $s+3$. Since $s \leq k+1, r \geq 5$, and $k \leq \frac{r-2}{3}$, it follows that $s+3 \leq r-k+1$. Finally note that the number of type A columns is at most $s$ so $|F| \leq r-k+1$.

If the columns in $F$ have exactly one non-zero entry in the first $k$ entries, then by \ref{zero}, $s$ is at most two. If $s=1$, then all the columns in $F-f^P$ are type C; if a column is of type A, then it is parallel to some element in $I_r$, a contradiction since $M$ is simple. Since the number of type C columns is at most $r-(s+k)$, it follows that $|F|$ is at most $r-k$. Now suppose that $s=2$. Note that all columns in $F-f^P$ are either type A or type C; a type B column contradicts that $M$ is simple. It is clear that the number of type A columns is at most two and the number of type C columns is $r-(2+k)$. However, if $F$ contains a type C column, then, by \ref{one}, the number of type A columns is at most one. Therefore $|F|$ is at most $\text{max}\{3,r-k\}$. Since $k \leq \frac{r-2}{3}$, it follows that $|F| \leq r-k$. Thus \ref{two} holds.

Observe that if there is a column $h^P$ in $Q-e^P$ that has its first $k$ entries equal to zero, then by \ref{zero}, the column $h^P$ has at most one non-zero root entry so $h^P$ has at most one non-zero entry. Since $M$ is simple, this is a contradiction. It follows that there are at most $2^k-1$ sets $F$ in $Q-e^P$, and $k$ of these have exactly one non-zero entry in the first $k$ entries. Therefore, by \ref{two}, $|E(M)| \leq k(r-k) + (2^k-1-k)(r-k+1)+ (r+1) = 2^k(r-k+1)$. 

If $e$ is $(k-1)$-loose, we choose the smallest non-negative integer $t$ such that $e$ is $t$-loose. By the same argument we obtain $|E(M)| \leq 2^t(r-t+1)$.  Since $t \leq k-1, r \geq 5$, and $k \leq \frac{r-2}{3}$, it follows that  $2^t(r-t+1) \leq 2^k(r-k+1)$.
\end{proof}

The following proposition demonstrates that the bound in Theorem \ref{size_bound_one_k_loose} is sharp.

\begin{proposition}
\label{sharp}
For $r \geq 2$ and $0 \leq k \leq r-2$, there exists a simple binary matroid of rank $r$ with a $k$-loose element, whose size is $2^k(r-k+1)$. 
\end{proposition}

\begin{proof}

Let $M$ be a binary matroid represented by the matrix $P = [I_r|Q]$ over $GF(2)$ such that $Q$ contains a column $e$ whose first $k$ entries are zero while the remaining entries are non-zero. We now describe the other columns of $Q$, demonstrate that $M$ has size $2^k(r-k+1)$, and show that the element corresponding to the column $e$ is $k$-loose in $M$. 
 
For a column $f$ in $Q$, we refer to the column obtained by removing the first $k$ entries of $f$ as the \textit{root vector} of $f$. The columns in $Q-e$ can be partitioned into disjoint sets $F$, where every column in $F$ has identical first $k$ entries. No column in $Q-e$ has all of its first $k$ entries equal to zero, so we have $2^k-1$ such sets $F$. 

For the sets $F$ where each column in $F$ has exactly one non-zero entry in its first $k$ entries, the root vectors of the columns in $F$ are the columns of the identity matrix of size $r-k$, so each such set $F$ has size $r-k$. For the remaining $2^k-(k+1)$ sets $F$, the root vectors of the columns in $F$ are the columns of the identity matrix of size $r-k$, along with the zero column, which has all of its entries equal to zero. The size of each of these sets $F$ is $r-k+1$. It follows that the size of $M$ is $2^k(r-k+1)$. 

Note that there is no zero column in $P$. Moreover, since $r \geq 2$ and $k \leq r-2$, all of the columns in $P$ are distinct, so $M$ is simple. Observe that if we add fewer than $r-k$ columns to $e$, then we cannot obtain the zero column so all circuits containing $e$ in $M$ have size greater than $r-k$. Therefore $e$ is $k$-loose in $M$. 
\end{proof}

\section{$k$-loose elements in $GF(q)$-matroids}

Zaslavsky and the author in \cite{sinzas} showed that a ternary matroid with no coloops that contains a $1$-loose element has linear size relative to its rank. In particular they proved the following.

\begin{theorem}
\label{ternary_one_aloose_intro}
For $r \geq 5$, let $M$ be a simple ternary matroid of rank $r$ containing a loose element and with no coloops. Then $|E(M)| \leq \lfloor\frac{41r-101}{2} \rfloor$ if $r > 10$. If $r \leq 10$, then $|E(M)| \leq \lfloor\frac{35r-35}{2} \rfloor$.
\end{theorem}

Furthermore, they proved that if a cosimple ternary matroid contains two $1$-loose elements, then the rank of the matroid is bounded. This result was shown to hold for all $GF(q)$-matroids. 

\begin{theorem}
\label{two_almost_loose}
Let $M$ be a simple $GF(q)$-matroid with no coloops that has two $1$-loose elements $e$ and $f$. Then $r(M) \leq 2q$, or $\{e,f\}$ is a cocircuit of $M$.
\end{theorem}

We extend Theorem \ref{two_almost_loose} to $k$-loose elements and prove Theorem \ref{rank_bound_two_loose}. The following is the restatement of Theorem \ref{rank_bound_two_loose}.

\begin{theorem*}
Let $M$ be a simple $GF(q)$-matroid with no coloops that has two $k$-loose elements $e$ and $f$. Then $r(M) \leq (q+1)(k-1)+2q$, or $\{e,f\}$ is a cocircuit of $M$.
\end{theorem*}

\begin{proof}
We may assume that $\{e,f\}$ is not a cocircuit of $M$. Therefore, there is a basis $B$ of $M$ that is disjoint from $\{e,f\}$. Note that, in the standard $GF(q)$-representation $P = [I_r|Q]$ of $M$ with respect to $B$, the column $e^P$ corresponding to $e$ and the column $f^P$ corresponding to $f$ have at most $k$ zeroes. By row and column scaling, we may assume that $e^P$ has all ones except that the first $k$ entries may be zero. 

Let $f^{P'}$ be the column obtained by removing the first $k$ entries of $f^P$. Suppose $f^{P'}$ has $k+2$ non-zero entries that are all equal. Then $f^P, e^P$, and at most $r-(k+2)$ columns from $I_r$ can be combined linearly over $GF(q)$ to obtain a column that has all its entries equal to zero. This would create a circuit of $M$ that contains $\{e,f\}$, and has size at most $r-k$, a contradiction. Therefore, $f^{P'}$ has at most $k+1$ non-zero entries that are equal. Since the field $GF(q)$ has $q-1$ non-zero elements, it follows that $f^{P'}$ has at most $(k+1)(q-1)$ non-zero entries. Additionally, $f^{P'}$ has at most $k$ zeroes so it has at most $kq+q-1$ entries. Therefore, $f^P$ has at most $kq+q-1+k$ entries so the rank of $M$ is at most $(q+1)(k-1)+2q$.
\end{proof}

\begin{corollary}
Let $M$ be a simple and cosimple $GF(q)$-matroid that has two $k$-loose elements. Then $r(M) \leq (q+1)(k-1)+2q$. 
\end{corollary}

The following is a restatement of Corollary \ref{rank_bound_paving}. It follows from Theorem \ref{rank_bound_two_loose} and provides the bound whose existence was established by Rajpal \cite[Proposition 8]{rajpal2}.

\begin{corollary*}
Let $M$ be a simple $GF(q)$-matroid with no coloops that is $k$-paving. Then $M$ is a circuit, or $r(M) \leq (q+1)(k-1)+2q$. 
\end{corollary*}

\begin{proof}
Note that every element of $M$ is $k$-loose. If $r(M)>(q+1)(k-1)+2q$, it follows from Theorem \ref{rank_bound_two_loose} that every pair of elements of $M$ is in a cocircuit of size two. It follows that $M$ is a circuit. 
\end{proof}

The following special cases of Corollary \ref{rank_bound_paving} are noteworthy. The binary case was shown by Rajpal \cite[Proposition 6]{rajpal2}. 

\begin{corollary}
\label{binary_rajpal}
Let $M$ be a simple binary matroid with no coloops that is $k$-paving. Then $M$ is a circuit, or $r(M) \leq 3k+1$.  
\end{corollary}

\begin{corollary}
Let $M$ be a simple ternary matroid with no coloops that is $k$-paving. Then $M$ is a circuit, or $r(M) \leq 4k+2$.  
\end{corollary}

\section{$k$-paving binary matroids}

Theorem \ref{size_bound_one_k_loose} bounds the size of a binary matroid with a $k$-loose element and rank at least $3k+2$. 
By Corollary \ref{binary_rajpal}, the rank of binary $k$-paving matroids is at most $3k+1$. Theorem \ref{size_bound_paving} bounds the size of these matroids. The following is its restatement.

\begin{theorem*}
Let $M$ be a simple binary $k$-paving matroid of rank $r$, where $M$ is not a circuit and $r \geq k+4$, and let $t = \lfloor \frac{3k+1-r}{2} \rfloor$. Then $r(M) \leq 3k+1$, and $$|E(M)| \leq (r+1) + \sum_{i=0}^t \binom{k}{i}.$$
\end{theorem*}

\begin{proof}
By Corollary \ref{binary_rajpal}, it is evident that $r(M) \leq 3k+1$. Suppose that $M$ has a circuit $C$ of size $r-k+1$. For an element $e$ of $C$, let $B$ be a basis of $M$ containing $C-e$. Consider the standard binary representation $P = [I_r|Q]$ of $M$ with respect to the basis $B$. 
\begin{sublemma}
\label{new_one}
Let $h$ be the sum of $m$ columns in $Q$. Then $h$ has at most $k+(m-1)$ zeroes. In particular, for any two columns $a$ and $b$ in $Q$, the sum $a+b$ has at most $k+1$ zeroes.
\end{sublemma}

Assume that $h$ has at least $k+m$ entries that are zero. Then, at most $r-(k+m)$ columns in $I_r$ can be added to $h$ to obtain the column with all entries equal to zero. This would result in a circuit of size at most $r-k$ in $M$, a contradiction. Therefore \ref{new_one} holds.

Note that the column $e^P$ in $Q$ corresponding to the element $e$ of $M$ has exactly $k$ zeroes. By reordering rows and columns of $P$, we may assume that the first $k$ entries of $e^P$ are zero. We denote the set of indices of the zero entries of a column $f^P$ in $Q$ by $\text{zeroes}(f^P)$. The entries of $f^P$ excluding the first $k$ entries are the \textit{root entries} of $f^P$. Let $g^P$ be a column in $Q-e^P$ such that $|\text{zeroes}(e^P) \cap \text{zeroes}(g^P)| = l$ is maximum. 

\begin{sublemma}
\label{new_two}
$r \leq 3k+1-2l.$
\end{sublemma}

Observe that $g^P$ has at most $k$ entries that are zero; otherwise, $M$ has a circuit of size less than $r-k$. By \ref{new_one}, it follows that $g^P$ has at most $k-l$ root entries that are zero. Furthermore, by \ref{new_one}, at most $k+1-l$ root entries of $g^P$ are equal to one. Therefore, $g^P$ has at most $k + (k-l) + (k+1-l) = 3k+1-2l$ entries. Thus, \ref{new_two} holds.

Let $F$ be a set of all columns in $Q-e^P$ such that, for any pair $f^P, g^P$ in $F$, the sum $f^P + g^P$ has its first $k$ entries equal to zero. Evidently, the first $k$ entries of all the columns in $F$ are identical.

\begin{sublemma}
\label{new_three}
$|F| \leq 1.$
\end{sublemma}

Suppose not, and let $a^P$ and $b^P$ be two columns in $F$. Note that the first $k$ entries of $a^P+b^P$ are all zero, so by \ref{new_one}, at most one root entry of $a^P+b^P$ is zero. It follows that $a^P+b^P+e^P$ has at most one non-zero entry. Therefore $M$ has a circuit of size at most four. Since $r \geq k+4$ and every circuit is of size greater than $r-k$, this is a contradiction. Thus, \ref{new_three} holds.

\begin{sublemma}
\label{new_four}
$|E(M)| \leq (r+1) + \sum_{i=0}^l \binom{k}{i}$.
\end{sublemma}

Observe that each column in $Q-e^P$ has at most $l$ zeroes among the first $k$ entries. It follows that there are at most $\sum_{i=0}^l \binom{k}{i}$ sets $F$ in $Q-e^P$. Therefore, by \ref{new_three}, $|E(M)| \leq (r+1) + \sum_{i=0}^l \binom{k}{i}$ so \ref{new_four} holds.  

By \ref{new_two}, $l \leq \lfloor \frac{3k+1-r}{2}\rfloor = t$ so by \ref{new_four}, the result follows. 

If $M$ has no circuit of size $r-k+1$, let $w$ be the smallest number such that $M$ has a circuit of size $r-w+1$. Clearly, $w < k$. By the same argument, it follows that $|E(M)| \leq (r+1) + \sum_{i=0}^s \binom{w}{i}$, where $s = \lfloor \frac{3w+1-r}{2} \rfloor$. Since $w < k$, it follows that $s \leq t$, and  $|E(M)| \leq (r+1) + \sum_{i=0}^t \binom{k}{i}$.
\end{proof}

We apply Theorem \ref{size_bound_paving} to $3$-paving binary matroids and prove Corollary \ref{3_paving}.

\begin{proof}[Proof of Corollary \ref{3_paving}]
It is clear that $r(M) \leq 10$. Note that if $r(M)$ is $7$, then $t$ equals $1$ so by Theorem \ref{size_bound_paving}, $|E(M)|$ is at most $12$. Similarly, if $r(M) = 8 , 9, $ and $10$, then $|E(M)|$ is at most $13, 11,$ and $12$, respectively. 
\end{proof}

\section*{Acknowledgment}

The author thanks James Oxley for helpful suggestions.

\end{document}